\documentclass[11pt]{amsart}

\title{Combinatorial solutions to the reflection equation}
\author{Agata Smoktunowicz}
\author{Leandro Vendramin}
\author{Robert Weston}

\address[A. Smoktunowicz]{
School of Mathematics, The University of Edinburgh, James Clerk Maxwell
Building, The Kings Buildings, Mayfield Road EH9 3JZ, Edinburgh}
\email{A.Smoktunowicz@ed.ac.uk}

\address[L. Vendramin]{IMAS--CONICET and Departamento de Matem\'atica, FCEN, Universidad de Buenos
Aires, Pabell\'on~1, Ciudad Universitaria, C1428EGA, Buenos Aires, Argentina, and NYU-ECNU Institute of Mathematical Sciences at NYU Shanghai, 3663 Zhongshan Road North, Shanghai, 200062, China}
\email{lvendramin@dm.uba.ar}

\address[R. Weston]{Dept of Mathematics, Heriot-Watt University,
Edinburgh EH14 4AS, UK, and The Maxwell Institute for Mathematical Sciences, Edinburgh}
\email{R.A.Weston@hw.ac.uk}

\usepackage{amssymb}
\usepackage{amsmath}
\usepackage{mathtools}
\usepackage{pdfpages}
\usepackage{changes}
\usepackage{tikz-cd}

\newcommand{\Z}{\mathbb{Z}}

\newcommand{\C}{\mathbb{C}}

\newcommand{\Aut}{\operatorname{Aut}}

\newcommand{\id}{\mathrm{id}}
\newcommand{\Soc}{\mathrm{Soc}}

\newcommand{\mr}{\mathfrak{r}}
\newcommand{\mk}{\kappa}

\makeatletter
\numberwithin{equation}{section}
\numberwithin{figure}{section}
\numberwithin{table}{section}
\newtheorem{thm}{Theorem}[section]
\newtheorem{lem}[thm]{Lemma}
\newtheorem{cor}[thm]{Corollary}
\newtheorem{pro}[thm]{Proposition}
\newtheorem{defn}[thm]{Definition}

\newtheorem{notation}[thm]{Notation}
\newtheorem{example}[thm]{Example}

\newtheorem{rem}[thm]{Remark}
\newtheorem{exa}[thm]{Example}

\makeatother

\begin{document}

\begin{abstract}
	We use ring-theoretic methods and methods from the theory of
	skew braces to produce set-theoretic solutions to the reflection equation. 
	We also use set-theoretic solutions to construct solutions of the parameter dependent reflection equation.
\end{abstract}

\maketitle

\section*{Introduction}

Following Drinfeld's suggestion in~\cite{MR1183474}, the study of set-theoretic
solutions to the Yang--Baxter equation (YBE) started in the seminal papers of
Etingof, Schedler and Soloviev~\cite{MR1722951} and Gateva--Ivanova and Van den
Bergh~\cite{MR1637256}. Since then, different aspects of this combinatorial
problem have been developed~\cite{tatiana,MR2885602,MR1769723,
MR2132760,MR1809284} and several interesting connections have been found.  In
pure mathematics, some of these connections are with braid and Garside
groups~\cite{MR2764830,MR3374524}, (semi)groups of
I-type~\cite{MR1637256,MR2301033}, matched pairs of
groups~\cite{MR1769723,MR2024436}, Artin-Schelter regular
algebras~\cite{MR2927367}, Jacobson radical rings and
generalizations~\cite{MR3177933,MR2278047}, regular subgroups and Hopf--Galois
extensions~\cite{MR3763907}, affine manifolds~\cite{MR3291816},
orderability~\cite{MR3815290,MR3572046} and factorizable
groups~\cite{MR1178147}. In mathematical physics, the connections include Yang--Baxter maps
\cite{MR1995883}, discrete integrable systems, cellular automata, crystals and tropical geometry 
(see \cite{MR2892302} and references therein). As a result
of all these relationships, there has been intensive study of set-theoretic
solutions to the quantum Yang--Baxter equation. 

Recall that a pair $(X,r)$, where $X$ is a set and $r\colon X\times X\to
X\times X$ is a bijective map, is a set-theoretic solution to the quantum Yang--Baxter
equation if\footnote{In the physics literature this $r$ is usually denoted by $\check{R}$.}
\begin{align*}
	(r\times\id)(\id\times r)(r\times\id)=(\id\times r)(r\times\id)(\id\times r).
\end{align*}
The solution $(X,r)$ is said to be non-degenerate if it is possible to write 
\[
r(x,y)=(\sigma_x(y),\tau_y(x))
\]
for bijective maps $\sigma_x,\tau_x\colon X\to X$ for all $x\in X$. As usual,
the solution $(X,r)$ will be called involutive if $r^2=\id_{X\times X}$. 

If $R$ is an associative ring, the operation $x\circ y=x+y+xy$ is always associative with
neutral element $0_R$. In the case when this operation $\circ$ turns $R$ into a group, then $R$
is a Jacobson radical ring~\cite{MR0012271}. As it was observed by
Rump~\cite{MR2278047}, Jacobson radical rings produce highly
non-trivial set-theoretic solutions to the Yang--Baxter equation. More precisely, if $R$ is a Jacobson 
radical ring, then 
\begin{align*}
	&r\colon R\times R\to R\times R,
	&&r(x,y)=(xy+y,(xy+y)'xy),
\end{align*}
is a non-degenerate involutive solution to the
Yang--Baxter equation, 
where $(xy+y)'$ denotes the inverse of the element $xy+y$ with respect to the
Jacobson circle operation $\circ$.

Rump observed that Jacobson radical rings can be generalized to braces. With
braces one produces non-degenerate involutive solutions very similar to those
coming from radical rings. These new solutions are universal in
the sense that each non-degenerate involutive solution is isomorphic to the 
restriction of a solution constructed from a brace. To study non-involutive
solutions one replaces braces by skew braces~\cite{MR3647970}. 
Skew braces still share several
properties with braces and of course with Jacobson radical rings, so 
techniques and tools from ring theory are available to study arbitrary
set-theoretic solutions.  

Recently, there has been considerable interest in the reflection equation, which
first appeared in the study of quantum scattering on the half-line by
Cherednik~\cite{MR774205}. The role of parameter dependent solutions to the reflection equation in the description of quantum-integrable systems with open boundaries was formulated by Sklyanin \cite{MR953215}.
Just as for the theory of the Yang--Baxter equation,
it also turns out to be interesting to study a combinatorial version of the
reflection equation.
For a map $k:X\rightarrow X$, and $(X,r)$ as above, this combinatorial reflection equation is \begin{equation*}
		r(\id\times k)r(\id\times k)=(\id\times k)r(\id\times k)r.
	\end{equation*}
This set-theoretical equation together with the first examples of solutions first 
appeared in the work of Caudrelier and Zhang~\cite{MR3207925}. A more systematic study and a classification in a slightly different setting (i.e. not inspired by soliton interactions but by maps appearing in integrable discrete systems) appeared in~\cite{MR3030177}.
Other solutions were also considered and used by Kuniba, Okado and Yamada within the context of cellular automata \cite{kuniba05}.

In this work we use ring-theoretic methods, and more generally methods coming
from the theory of braces, to produce families of new solutions to the
reflection equation. Our purely combinatorial approach is
different to that of~\cite{DeCommer}, where actions of skew braces are used to
produce reflections.

\smallskip
The paper is organized as follows. In Section~\ref{preliminaries} we give the
main definitions and examples and then prove that for finite, non-degenerate, involutive set-theoretic solutions to
the YBE one only needs to check one of the coordinates to prove that a certain
map is a reflection.  Theorem~\ref{thm:invariant} shows that each map invariant
under the action of the permutation group of a finite non-degenerate involutive solution yields a reflection;
this result easily produces several reflections.  
Section~\ref{braces} contains several reflections constructed from the theory of left braces
and Jacobson radical rings. 
Section~\ref{factorization} explores  
reflections associated with the solutions of the YBE  constructed by Weinstein and Xu from factorizable groups. 
In Section~\ref{paramdepsolns} we show a way of introducing parameter dependence
in both $r$ and $k$ to yield solutions of the respective parameter dependent
quantum Yang--Baxter and reflection equations.

\section{Preliminaries}
\label{preliminaries}

A set-theoretic solution to the YBE is a pair $(X,r)$, where $X$ is a set and
$r\colon X\times X\to X\times X$ is a bijective map such that
\[
	r_1r_2r_1=r_2r_1r_2,
\]
where $r_1=r\times\id$ and $r_2=\id\times r$. 
By convention, we write
\[
	r(x,y)=(\sigma_x(y),\tau_y(x)).
\]
If the solutions is said to be non-degenerate, then $\sigma_x,\tau_x\colon X\to
X$ are assumed to be bijective. The solution $(X,r)$ is finite if $X$ is finite
and it is involutive if $r^2=\id$.

\begin{rem}
	\label{rem:involutive}
	If $(X,r)$ is a non-degenerate involutive solution, then
	\[
		\sigma_{\sigma_x(y)}(\tau_y(x))=x,\quad
		\tau_{\tau_y(x)}(\sigma_x(y))=y.
	\]
	for all $x,y\in X$. In particular, if furthermore $(X,r)$ is non-degenerate, then
	\[
		\tau_y(x)=\sigma^{-1}_{\sigma_x(y)}(x),\quad
		\sigma_x(y)=\tau^{-1}_{\tau_y(x)}(y).
	\]
	for all $x,y\in X$.
\end{rem}

Let us first recall the basic definitions from~\cite{MR3030177}.

\begin{defn}
	Let $(X,r)$ be a non-degenerate solution. We say that 
	a map $k\colon X\to X$ is 
	a \emph{reflection} of $(X,r)$ if
	\[
		r(\id\times k)r(\id\times k)=(\id\times k)r(\id\times k)r.
	\]
	The reflection $k$ is said to be \emph{involutive} if $k^2=\id$.
\end{defn}

If $X=\{1,\dots,n\}$ and $k\colon X\to X$, we use the one-line notation for $k$
which is simply the string $k(1)k(2)\cdots k(n)$. For example, the identity
would be $\id=123\cdots n$.

\begin{example}
	\label{exa:3,2}
	Let $X=\{1,2,3\}$ and $r(x,y)=(\varphi_x(y),\varphi_y(x))$, where
	$\varphi_1=\varphi_2=\id$ and $\varphi_3=(12)$. There are five 
	reflection maps:
	\[
		k_1=123=\id,\quad
		k_2=113,\quad
		k_3=213,\quad
		k_4=223,\quad
		k_5=333.
	\]
	Moreover, $k_j$ is involutive if and only if $j\in\{1,3\}$.
\end{example}

\begin{example}
	Let $X$ be a set and $\sigma,\tau\colon X\to X$ be permutations such that
	$\sigma\tau=\tau\sigma$.  Then $(X,r)$, where $r(x,y)=(\sigma(y),\tau(x))$
	is a solution to the YBE. Each map $k\colon X\to X$ that commutes with
	$\sigma\tau$ is a reflection map.
\end{example}

\begin{example}
	\label{exa:4,13}
	Let $X=\{1,2,3,4\}$ and $r(x,y)=(\sigma_x(y),\tau_y(x))$, where
	\begin{align*}
		\sigma_1=(34), &&\sigma_2=(1324), &&\sigma_3=(1423), &&\sigma_4=(12),\\
		\tau_1=(24),&&\tau_2=(1423),&&\tau_3=(1234),&&\tau_4=(13).
	\end{align*}
	There are ten reflection maps: 
	\begin{align*}
		& k_1=1144, && k_2=1212, && k_3=1234, && k_4=1331, && k_5=2143,\\
		& k_6=2233, && k_7=3412, && k_8=3434, && k_9=4224, && k_{10}=4321.
	\end{align*}
	Moreover, $k_j$ is involutive if and only if $j\in\{3,5,7,10\}$.
\end{example}

\begin{notation}
	Let $(X,r)$ be a solution. 
	For each $x,y\in X$ let 
	\begin{align*}
		t(x,y)=\sigma_{\sigma_x(y)}k(\tau_y(x)), && 
		u(x,y)=\tau_{k(\tau_y(x))}\sigma_x(y).
	\end{align*}
\end{notation}

\begin{lem}
	\label{lem:basic}
	Let $(X,r)$ be a non-degenerate solution to the YBE. The map $k\colon X\to X$ 
	is a reflection of $(X,r)$ if and only if
	\begin{align*}
		t(x,k(y))=t(x,y) \quad \text{and} \quad 
		u(x,k(y))=k(u(x,y))
	\end{align*}
	for all $x,y\in X$.
\end{lem}

\begin{proof}
	It is straightforward.
\end{proof}

The following theorem shows that for involutive solutions one needs to check
only one of the formulas of Lemma~\ref{lem:basic}. 

\begin{thm}	
	\label{thm:t}
	Let $(X,r)$ be a non-degenerate involutive solution to the YBE and let
	$k\colon X\to X$ be a map. Then $k$ is a reflection of $(X,r)$ if and only
	if 
	\begin{equation}
		\label{eq:t}
		t(x,y)=t(x,k(y))
	\end{equation}
	for all $x,y,\in X$.
\end{thm}

\begin{proof}
	One of the implications follows directly from Lemma~\ref{lem:basic}. Let us
	then assume that~\eqref{eq:t} holds for all $x,y\in X$. Let $x,y\in X$. We
	need to prove that 
	$u(x,k(y))=k(u(x,y))$. 
	Let 
	\begin{align}
		\label{eq:abde}
		a = \sigma_x(y), && b=\tau_y(x), && 
		d = \sigma_a(k(b)), && e=k(\tau_{k(b)}(a)).
	\end{align}
	A straightforward calculation shows that $\sigma_d(e)=t(a,k(b))$. Since
	$(X,r)$ is involutive, $\sigma_a(b)=x$ and $\tau_b(a)=y$. These facts and
	Equality~\eqref{eq:t} imply that 
	\[
		\sigma_d(e)=t(a,k(b))=t(a,b)=\sigma_x(k(y)).
	\]
	Since $(X,r)$ is non-degenerate, 
	$e=\sigma^{-1}_d\sigma_x(k(y))$ and hence 
	\[
		k(u(x,y))=k\left(\tau_{k(\tau_y(x))}\sigma_x(y)\right)=\sigma^{-1}_d\sigma_x(k(y))
	\]
	by~\eqref{eq:abde} and the definition of $u(x,y)$.
	Now let 
	\[
		g=u(x,k(y))=\tau_{k(\tau_y(x))}\sigma_x(k(y)),\quad
		f=\sigma_{\sigma_x(k(y))}k(\tau_{k(y)}(x)).
	\]
	Applying $r^2=\id$ to the pair $\left(\sigma_x(k(y)),k(\tau_y(x))\right)$ one gets 
	\[
		\sigma_f(g)=\sigma_x(k(y)),\quad
		\tau_g(f)=k(\tau_y(x)).
	\]
	Since $d=t(x,y)=t(x,k(y))=f$, it follows that 
	\[
		u(x,k(y))=g=\sigma^{-1}_f\sigma_x(k(y))=\sigma^{-1}_d\sigma_x(k(y)).
	\]
	Then Lemma~\ref{lem:basic} implies the claim.
\end{proof}

The \emph{involutive Yang--Baxter group} of an involutive
non-degenerate solution $(X,r)$ is the group $\mathcal{G}(X,r)$ generated by
$\{\sigma_x:x\in X\}$. Clearly this group acts on $X$ by evaluation. A map
$k\colon X\to X$ is said to be $\mathcal{G}(X,r)$-equivariant if $k(gx)=gk(x)$
for all $g\in\mathcal{G}(X,r)$ and $x\in X$. Thus $k$ is
$\mathcal{G}(X,r)$-equivariant if and only if $k\sigma_x=\sigma_xk$ for
all $x\in X$.

\begin{thm}
	\label{thm:invariant}
	Let $(X,r)$ be an involutive non-degenerate solution to the YBE. Each 
	$\mathcal{G}(X,r)$-equivariant map $k\colon X\to X$ is a
	reflection of $(X,r)$.
\end{thm}

\begin{proof}
	Let $x,y\in X$ and $u=\sigma_x(k(y))$ and $v=\tau_{k(y)}(x)$.
	Using Remark~\ref{rem:involutive} and that $k$ is
	$\mathcal{G}(X,r)$-equivariant, we write
	\begin{align*}
		rk_2rk_2(x,y)
		=(\sigma_u(k(v)),\beta)
		=(k(\sigma_u(v)),\beta)
		=(k(x),\beta)
	\end{align*}
	for some $\beta\in X$. Similarly, we write 
	\begin{align*}
		k_2rk_2r(x,y)&=k_2rk_2(\sigma_x(y),\tau_y(x))
		=(\sigma_{\sigma_x(y)}k(\tau_y(x)),\gamma)
		=(k(x),\gamma)
	\end{align*}
	for some $\gamma\in X$. Hence the claim follows from Theorem~\ref{thm:t}.
\end{proof}


The converse of Theorem~\ref{thm:invariant} does not hold:  

\begin{example}
	\label{exa:4,19}
	Let $X=\{1,2,3,4\}$ and $r(x,y)=(\sigma_x(y),\tau_y(x))$, where 
	\begin{align*}
		\sigma_1=(12),&&\sigma_2=(1324),&&\sigma_3=(34),&&\sigma_4=(1423),\\
		\tau_1=(14),&&\tau_2=(1243),&&\tau_3=(23),&&\tau_4=(1342).
	\end{align*}
	There are then ten reflection maps:
	\begin{align*}
		& k_1=1133, && k_2=1221, && k_3=1234, && k_4=1414, && k_5=2143,\\
		&k_6=2244,&& k_7=3232, && k_8=3412, && k_9=4321, && k_{10}=4334.
	\end{align*}
	A direct calculation shows that $\mathcal{G}(X,r)$ is isomorphic to the
	dihedral group of eight elements. Moreover, $k_j$ is
	$\mathcal{G}(X,r)$-equivariant if and only if $j\in\{3,5\}$.
\end{example}

\section{Reflections and left braces}

A left brace is a triple $(A,+,\circ)$ such that $(A,+)$ is an abelian group,
$(A,\circ)$ is a group and 
\[
	a\circ (b+c)=a\circ b-a+a\circ c
\]
holds for all $a,b,c\in A$. If $A$ is a left brace, the multiplicative group
$(A,\circ)$ of $A$ acts by automorphism on the additive group $(A,+)$, i.e.
the map $\lambda\colon (A,\circ)\to\Aut(A,+)$ given by $a\mapsto \lambda_a$,
where $\lambda_a(b)=-a+a\circ b$, is a group homomorphism. For $a,b\in A$ 
one defines the operation
\[
	a*b=\lambda_a(b)-b=-a+a\circ b-b.
\]
Left braces produce
solutions to the YBE. The map 
\begin{gather}
	r_A\colon A\times A\to A\times A,
	\quad
	r_A(a,b)=(\lambda_a(b),\mu_b(a)),\label{eq:rdef}
	\shortintertext{where}
	\mu_b(a)=\lambda_a(b)'\circ a\circ b=\lambda_a(b)'*a+a,
\nonumber\end{gather}
where $x'$ denotes the inverse of the element $x$ with respect to the circle
operation, is an involutive non-degenerate solution to the YBE. 
We refer to~\cite{MR3824447} for an introduction to the theory of left braces.

A left ideal of a left brace $A$ is a subgroup $X$ of the additive group of $A$
such that $A*X\subseteq X$. An ideal $I$ of $A$ is a left ideal $I$ of $A$ such
that $a+I=I+a$ and $a\circ I=I\circ a$ for all $a\in A$.  The \emph{socle} of a
left brace $A$ is defined as 
\[
	\Soc(A)=\{a\in A:\lambda_a=\id\}=\{a\in A:a+b=a\circ b\text{ for all $b\in A$}\}
\]
and it is an ideal of $A$.

\begin{thm}
	\label{thm:general}
	Let $A$ be a left brace, $X\subseteq A$ be a subset such that the restriction 
	$r=r_A|_{X\times X}$ is a solution $(X,r)$ to the YBE and  $g\colon X\to Y$ be a
	map for some subset $Y\subseteq A$. Assume that there exists a map $X\times Y\to X$, $(x,y)\mapsto x\odot
	y$, such that
	\begin{equation}
		\label{eq:odot}
		\lambda_x(y\odot g(z))=\lambda_x(y)\odot g(z)
	\end{equation}
	for all $x,y,z\in X$. Let $f\colon X\to X$ be  
	$\mathcal{G}(X,r)$-equivariant and $k\colon X\to X$ be given by
	$k(x)=f(x)\odot g(x)$. Then $k$ is a reflection of $(X,r)$ if and only if 
	\begin{equation*}
		f(x)\odot g(\mu_{k(y)}(x))=f(x)\odot g(\mu_y(x))
	\end{equation*}
	for all $x,y\in X$.
\end{thm}

\begin{proof}
	Let $x,y\in X$ and
	\[
	(u,v)=rk_2(x,y)=(\lambda_x(k(y)),\mu_{k(v)}(x)).
	\]
	Then by
	Remark~\ref{rem:involutive} and using that~\eqref{eq:odot} holds and $f$ is
	$\mathcal{G}(X,r)$-equivariant, 
	\begin{align*}
	\lambda_uk(v)&=\lambda_u(f(v)\odot g(v))\\&=\lambda_u(f(v))\odot g(v)=f(\lambda_u(v))\odot g(v)=f(x)\odot g(v)
	\end{align*}
	as $r^2(x,k(y))=r(u,v)=(\lambda_u(v),\mu_v(u))=(x,k(y))$. 
	Then 
	\begin{align}
		\label{eq:rkrk}
		rk_2rk_2(x,y)&=(f(x)\odot g(v),\beta)
	\end{align}
	for some $\beta\in X$. 
	Let $a,b\in X$ be such that $k_2r(x,y)=k_2(a,b)=(a,k(b))$. As we did before, 
	\begin{align*}
		\lambda_a(k(b))&=\lambda_a(f(b)\odot g(b))\\
		&=\lambda_a(f(b))\odot g(b)=f(\lambda_a(b))\odot g(b)=f(x)\odot g(b)\\
	\end{align*}
	Hence 
	\begin{align}
		\label{eq:krkr}
		k_2rk_2r(x,y)=(f(x)\odot g(b),\gamma)
	\end{align}
	for some $\gamma\in X$. 
	Using~\eqref{eq:odot}, 
	\[
		f(x)\odot g(v)=f(x)\odot g(\mu_{k(y)}(x))=f(x)\odot g(\mu_y(x))=f(x)\odot g(b)
	\]
	and therefore the first coordinate of~\eqref{eq:rkrk} is equal to the first
	coordinate of~\eqref{eq:krkr}. Now the claim follows from Theorem~\ref{thm:t}.
\end{proof}

As an application of Theorem~\ref{thm:invariant} we construct reflections using
two-sided braces and left ideals. As was mentioned before, Rump observed
that two-sided braces are equivalent to Jacobson radical rings. This
equivalence is based on the following lemma of~\cite{MR2278047}. 
We provide a proof for completeness. 

\begin{lem}
	\label{lem:2sided}
	Let $A$ be a left brace. Then
	\begin{equation}
		\label{eq:(aob)*c}
		(a\circ b)*c=
		(a\circ b\circ a')*\lambda_a(c)+a*c
	\end{equation}
	for all $a,b,c\in A$. 
	Furthermore, if $A$ is two-sided, then 
	\begin{enumerate}
		\item $(a+b)*c=a*c+b*c$ for all $a,b,c\in A$, 
		\item $(-a)*b=-a*b$ for all $a,b\in A$, and 
		\item $\lambda_a(b*c)=\lambda_a(b)*c$ for all $a,b,c\in A$.
	\end{enumerate}
\end{lem}

\begin{proof}
	Let $a,b,c\in A$. 
	To prove~\eqref{eq:(aob)*c} we compute
	\begin{align*}
		(a\circ b)*c&=\lambda_{a\circ b}(c)-c=\lambda_{a\circ b\circ a'}\lambda_{a}(c)-c\\
		&=(a\circ b\circ a')*\lambda_a(c)+\lambda_a(c)-c
		=(a\circ b\circ a')*\lambda_a(c)+a*c.
	\end{align*}
	We now prove (1). Using the commutativity of the additive group of
	$A$, 
	\[
		(a+b)*c=\lambda_{a+b}(c)-c
		=-a-b+a\circ c-c+b\circ c-c
		=a*c+b*c
	\]
	for all $a,b,c\in A$. To prove (2) we use (1) to obtain 
	\[
		0=0*b=(a+(-a))*b=a*b+(-a)*b,
	\]
	from where (2) follows.	Now we prove (3). On the one hand,
	\begin{equation}
		\label{eq:(3)a}
		\lambda_a(b*c)=\lambda_a(\lambda_b(c)-c)=\lambda_{a\circ b\circ a'}\lambda_a(c)-\lambda_a(c)=(a\circ b\circ a')*\lambda_a(c).
	\end{equation}
	On the other, using (2) and~\eqref{eq:(aob)*c}, 
	\begin{align*}
		\lambda_a(b)*c&=(-a+a\circ b)*c=(-a)*c+(a\circ b)*c\\
		&=-a*c+(a\circ b\circ a')*\lambda_a(c)+a*c
		=(a\circ b\circ a')*\lambda_a(c).\qedhere
	\end{align*}
\end{proof}

\begin{lem}
	Let $A$ be a left brace and $X$ be a left ideal of $A$. The map 
	$r=r_A|_{X\times X}$ is a solution $(X,r)$ to the YBE.
\end{lem}

\begin{proof}
	Since $X$ is a left ideal, $(X,+)$ is a subgroup of $(A,+)$ and $(X,\circ)$
	is a subgroup of $(A,\circ)$.  Then
	$r_A(x,y)=(\lambda_x(y),\lambda_x(y)'\circ x\circ y)\in X\times X$ if 
	$x,y\in X$.
\end{proof}

Now we obtain several corollaries of Theorem~\ref{thm:general} in the case of
Jacobson radical rings. 

\begin{cor}
	\label{cor:k=f*g}
	Let $A$ be a two-sided brace, $X$ be a left ideal of $A$ and 
	$f,g\colon X\to X$ be maps such
	that $f$ is $\mathcal{G}(X,r)$-equivariant and 
	\begin{equation}
		\label{eq:Agata}
		f(x)*g(\mu_{k(y)}(x))=f(x)*g(\mu_y(x)) 
	\end{equation}
	for all $x,y\in X$. Then 
	$k\colon X\to X$, $k(x)=f(x)*g(x)$, is a reflection of $(X,r_X)$. 
\end{cor}

\begin{proof}
	It follows from Theorem~\ref{thm:general} with $x\odot y=x*y$. 
\end{proof}

\begin{cor}
    \label{cor:5}
	Let $A$ be a two-sided brace, $X$ be a left ideal of $A$ 
	and  $g\colon X\to X$ be a map such that
	$g(\mu_{k(y)}(x))=g(\mu_y(x))$ 
	for all $x,y\in X$.  Then $k\colon X\to X$,
	$k(x)=x*g(x)$ is a reflection of $(X,r_X)$.
\end{cor}

\begin{proof}
	It follows from Corollary~\ref{cor:k=f*g} with $f=\id$. 
\end{proof}

\begin{cor}
	\label{cor:6}
	Let $A$ be a two-sided brace and $a\in A$. Then $k(x)=x*a$ is a reflection
	of $(A,r_A)$.
\end{cor}

\begin{proof}
	Use Corollary~\ref{cor:k=f*g} with $f=\id$ and $g(x)=a$ for all
	$x\in A$.
\end{proof}

Reflections from Corollaries~\ref{cor:5} and~\ref{cor:6} are 
$\mathcal{G}(X,r)$-equivariant. 

\begin{cor}
	\label{cor:7}
	Let $A$ be a two-sided brace, $X$ be a left ideal of $A$ and  $g\colon X\to
	X$ be a map such that $g(\mu_{k(y)}(x))=g(\mu_y(x))$ for all $x,y\in X$.
	Then $k\colon X\to X$, $k(x)=x+x*g(x)$, is a reflection of $(X,r_X)$.
\end{cor}

\begin{proof}
	It follows from Theorem~\ref{thm:general} with $f=\id$ and $x\odot	y=x+x*y$.
\end{proof}

\begin{cor}
	\label{cor:Agata}
	Let $A$ be a two-sided brace, $X$ be an ideal of $A$
	and $g\colon A\to A$ be such that
	$g(A)\subseteq X$ and $g(a+x)=g(a)$ for all $a\in A$ and $x\in X$. Then
	$k(x)=x+x*g(x)$ is a reflection of $(A,r_A)$.
\end{cor}

\begin{proof}
	Note that $k(x)-x=x*g(x)\in X$ for all $x\in X$. Since 
	\[
		\mu_{k(y)}(x)=\lambda^{-1}_{\lambda_x(k(y))}(x)=\lambda^{-1}_{\lambda_x(y)}(x)=\mu_y(x)
	\]
	in the quotient $A/X$, it follows that 
	$\lambda_x(k(y))=\lambda_x(y)+z$ for some $z\in X$. Hence 
	\[
		g(\mu_{k(y)}(x))=g(\mu_y(x)+z)=g(\mu_y(x))
	\]
	and the claim follows from Theorem~\ref{thm:general}
	with $f=\id$ and
	$a\odot b=a+a*b$.
\end{proof}

Reflections from Corollaries~\ref{cor:7} and~\ref{cor:Agata} do not seem to be
necessarily $\mathcal{G}(X,r)$-equivariant. 

\begin{cor}
	Let $A$ be a two-sided brace, $X$ be an ideal of $A$, $f\colon A\to A$ be 
	$\mathcal{G}(X,r)$-equivariant and such that $f(x)-x\in X$ for all $x\in A$
	and $g\colon A\to A$ be such that $g(A)\subseteq X$ and $g(a+x)=g(a)$ for
	all $a\in A$ and $x\in X$. Then $k(x)=f(x)+f(x)*g(x)$ is a reflection of
	$(A,r_A)$.
\end{cor}

\begin{proof}
	It is similar to the proof of Corollary~\ref{cor:Agata}. We only need to write in the first line
	$k(x)-x=f(x)*g(x)\in X$ for all $x\in X$.
\end{proof}

For some type of solutions $(X,r)$ all the reflection maps are 
as those of the previous corollaries as the following example shows:

\begin{example}
Let $(A, +, *  )$ be a nilpotent algebra over a field  of characteristic two
and let $z\in A$. Then
\[
S=\{z*r, z*r+z: r\in A\}
\]
is a nilpotent ring.
Let 
\[
X=\{\lambda _{s}(z): s\in S\}
\] 
and let $(X, r)$ be the associated
solution to the YBE. We will show that  if $a, b\in X$ then $b=\lambda _{a}(c)$ for some $c\in A$. Indeed, if $a=\lambda _{s}(z)$, for $s\in S$ 
then $b=\lambda _{s}(\lambda _{s^{-1}}(b))$ and  $\lambda _{s^{-1}}(b)\in X$, thus 
 $\lambda
_{s^{-1}}(b)=z*r+z$ for some $r\in R$.  It follows that  $b=\lambda _{s}(z*d+z)=a*d+a$, consequently any function $k: X\rightarrow X$ 
can be written as $k(x)=x*g(x)+x$ for some $g(x)\in
Az$. Therefore any reflection $k$ on $(X,r)$ satisfies assumptions of
Theorem~\ref{thm:general} with $x\odot y=x*y+x$ for $x,y\in X$ and $Y=Az$. In particular $k(x)$ is a solution of the reflection
equation on $(X,r)$  if and only if  $k(x)=x*g(x)+x$ for some $g:X\rightarrow Az$  such that   $g(\tau _{y}(x))=g(\tau _{k(y)}(x))$ for all $x,y\in
X$.
\end{example}

\section{Other reflections}

The following modification of Theorem~\ref{thm:general} is sometimes useful for describing solutions of the  reflection equation related to two-sided braces.

\begin{thm}
	\label{thm:aggi}
	Let $A$ be a two-sided brace, $X\subseteq A$ be a subset such that the restriction 
	$r=r_A|_{X\times X}$ is a solution $(X,r)$ to the YBE, 
	$g_1,\dots,g_n\colon X\to X$, 
	$f_1,\dots,f_n\colon X\to X$ be 
	$\mathcal{G}(X,r)$-equivariant 
	and $k\colon X\to X$ be given by
	\[
	k(x)=x+\sum_{i=1}^{n}f_{i}(x)\odot g_{i}(x),
	\]
	where the operation $\odot$ is as in Theorem~\ref{thm:general}. 
	Then $k$ is a reflection of $(X,r)$ if and only if 
	\begin{equation*}
		\sum_{i=1}^{n}f_i(x)\odot g_i(\mu_{k(y)}(x))=\sum_{i=1}^{n}f_i(x)\odot g_i(\mu_y(x))
	\end{equation*}
	for all $x,y\in X$.
\end{thm}
\begin{proof}
 It is similar to the proof of  Theorem~\ref{thm:general}.
\end{proof}

For a map $k\colon X\to X$, 
write $k_1=k\times\id$ and $k_2=\id\times k$. 

\begin{pro}
	Let $(X,r)$ be a non-degenerate involutive solution. Then $k_2r=rk_2$ if
	and only if $k=\id$. 
\end{pro}

\begin{proof}
	We only need to prove the non-trivial implication. The assumption is
	equivalent to $\sigma_x(y)=\sigma_x(k(y))$ and
	$k(\tau_y(x))=\tau_{k(y)}(x)$ for all $x,y\in X$. Then
	\[
	k(\tau_y(x))=\tau_{k(y)}(x)=\sigma^{-1}_{\sigma(k(y))}(x)=\sigma^{-1}_{\sigma(y)}(x)=\tau_y(x)
	\]
	and the claim follows since $(X,r)$ is non-degenerate.
\end{proof}

\begin{pro}
	\label{pro:Agata1}
	Let $(X,r)$ be a non-degenerate solution and $k\colon X\to X$ be such that
	either $k_2r=r^{-1}k_1$ or $k_1r^{-1}=rk_2$. Then $k$ is a reflection of
	$(X,r)$.
\end{pro}

\begin{proof}
	Let us first assume that $k_2r=r^{-1}k_1$. Then $rk_2rk_2=rr^{-1}=k_1k_2$
	and $k_2rk_2r=k_2rr^{-1}k_1=k_2k_1$. Since $k_1$ and $k_2$ commute, the
	claim follows. Similarly one proves the case where $k_1r^{-1}=rk_2$.
\end{proof}

Proposition~\ref{pro:Agata1} of course applies in the particular case of
involutive solutions.

\begin{pro}
	\label{pro:Agata2}
	Let $(X,r)$ be a non-degenerate solution and $k\colon X\to X$ be such that
	$k_2r=rk_1$ and $k_1r=rk_2$.  
	Then $k$ is a reflection of $(X,r)$.
\end{pro}

\begin{proof}
	We compute 
	$k_2rk_2r=rk_1k_2r=r^2k_2k_1=r^2k_1k_2$
	and similarly $rk_2rk_2=r^2k_1k_2$.
\end{proof}

Let $(X,r)$ be a non-degenerate involutive solution to the YBE; it is known that $X$ is a subset of some brace $A$ such that $X$ generates $A$ as an additive group and  restriction $r = r_{A}|_{X\times X}$ is a solution $(X,r)$ to the YBE. This statement is implicit in \cite{MR2278047} and explicit in Theorem 4.4 of \cite{MR3177933}.
In particular $A$ can be taken to be the structure group of $(X,r)$ (see \cite{MR3177933} for more information). Moreover, if $X$ is finite, $A$ can be assumed to be finite.
	
\begin{thm}
	\label{thm:k2r=rk1}
  Let $A$ be a left brace and $X\subseteq A$ be a subset such that the restriction $r = r_{A}|_{X\times X}$  is a solution $(X,r)$ to the YBE. Assume that $X$ generates the additive group of $A$. Let $k\colon
	X\to X$ be a map.  The following statements are equivalent:
	\begin{enumerate}
		\item $k_2r=rk_1$.
		\item $k_1r=rk_2$.
		\item $k(x)-x\in\Soc(A)$ for all $x\in X$ and $k$ is $\mathcal{G}(X,r)$-equivariant.
	\end{enumerate}
\end{thm}

\begin{proof}
	Let us first prove that $(1)$ and $(2)$ are equivalent. Assume first that
	(2) holds.  This means that $\lambda_x(k(y))=k(\lambda_x(y))$ and
	$\mu_{k(y)}(x)=\mu_y(x)$ hold for all $x,y\in X$. Since 
	\[
	\lambda^{-1}_{k(\lambda_x(y))}(x)=\lambda^{-1}_{\lambda_x(k(y))}(x)=\mu_{k(y)}(x)=\mu_y(x)=\lambda^{-1}_{\lambda_x(y)}(x),
	\]
	holds for all $x,y\in X$, by writing $y=\lambda^{-1}_x(z)$ one concludes
	that $\lambda^{-1}_{k(z)}(x)=\lambda^{-1}_z(x)$ for all $x,z\in X$. Now 
	\[
	\mu_y(k(x))=\lambda^{-1}_{\lambda_{k(x)}(y)}(k(x))=\lambda^{-1}_{\lambda_x(y)}(k(x))=k(\lambda^{-1}_{\lambda_x(y)}(x))=k(\mu_y(x)) 
	\]
	and hence (1) holds. The implication $(1)\implies(2)$ is similar.

	Now we prove that (1) and (2) imply (3). 
	Let $x,z \in X$ and $y \in X$ be such that $z = \lambda _{x}(y)$. By (2), $\lambda_{z}(y)=\lambda _{k(z)}(y)$  for every $z,y\in X$. Since $z*x = k(z)*x$ and $X$ generates additively the group $A$, it follows that $z *g = k(z)*g$ for all $g \in A$. Thus $\lambda _{z}=\lambda _{k(z)}$ and hence, in the quotient $G(X,r)/Soc(G(X,r))$, one has $z = k(z)$. Notice that this implies $z'=k(z)'$ in $G(X,r)/Soc(G(X,r))$.  Moreover, 
	\begin{align*}
		k(z*x+x)&=k\mu_y(x)=\mu_y(k(x))\\
		&=\lambda_{k(x)}(y)'*k(x)+k(x)\\
		&=\lambda_x(y)'*k(x)+k(x)=z*k(x)+k(x).
	\end{align*}

	Finally we prove that $(3)\implies(2)$. We need to show that both 
	\[
		\lambda_x(k(y))=k(\lambda_x(y)),\quad 
		\mu_{k(y)}(x)=\mu_y(x)
	\]
	hold for all
	$x,y\in X$.  
	
	Notice that $\mu_{k(y)}(x)=\mu_y(x)$ holds because $k(y)-y\in
	\Soc(A)$ (by the definition of $\mu $). By the definition
	of $k$ we have $k(\lambda_x(y))=\lambda _{x}(k(y))$, which concludes the
	proof. 
\end{proof}

The following example fits in the context of Theorem ~\ref{thm:k2r=rk1}.

\begin{exa}
Let $(A, +, *)$ be a nilpotent ring and $(A, +, \circ )$ be the corresponding
left brace, so $a\circ b=a*b+a+b$ for $a, b\in A$.  Let  $(A, r_{A})$ be the
associated solution of the YBE and $k(x) = x +
x^{n-1}$. Assume that $A^n=0$. Then 
\begin{align*}
&k(x)-x=x^{n-1}\in Soc (A),\\
&k(\lambda _{x}(y))=(x*y+y)+(x*y+y)^{n-1}=x*y+y+y^{n-1},\\
&\lambda _{x}(k(y))=\lambda _{x}(y+y^{n-1})=x*y+y+y^{n-1}.
\end{align*}
and therefore $r_Ak_{2} = k_{2}r_A$ by
Theorem~\ref{thm:k2r=rk1}.
\end{exa}

The following example shows that there are reflections that are {\it not} of the type covered by 
Theorem~\ref{thm:k2r=rk1}.

\begin{exa}
Let $A$ be a ring  generated by one generator $b$ subject to relation $b^4=0$, so $A$
is a nilpotent ring. Suppose that $x+x=0$ for every $x\in A$. Let $X = \{b + bf : f \in A\}$, and let $r$ be a
restriction of the Yang--Baxter map associated to $A$. Observe that any element $x\in X$  can be written 
 as $x=b+i \cdot b^{2}+j\cdot b^{3}+ \langle b^{4}\rangle$ for some $i , j \in \Z_{2}$ - for simplicity we will write $x=b+i\cdot b^{2}+j\cdot b^{3}$. Let $x, y\in X$; then  
$x = b + ib^2 + jb^{3}$ and $y = b + mb^2 + nb^{3}$ for some 
$i,j, m,n \in \Z_2$.  By using formula \eqref{eq:rdef} for the map $r$ associated to the brace $A$ we get $r(x, y)=(xy+y, zx+x)$ where $z(xy+y)+z+(xy+y)=0$. Therefore, 
\[
r(x,y)=(b + (m+1)b^2 + (i+m+n)b^{3}, b + (i+1)b^2 + (m+i+j)b^{3}).
\]  
 Define $k(x)=x+x^{2}$ for every $x\in X$. It follows that 
\[
k(b + ib^2 + jb^{3})=b + (i+1)b^2 + jb^{3}.\]

Then
\begin{align*}
k_{2}rk_{2}r(x,y)&=( b + (i+1)b^2 + (j+1)b^{3},  b + (m+1)b^2 +
(n+1)b^{3})\\
&=rk_{2}rk_{2}(x,y).
\end{align*}
Note that $k(x)-x=x^{2}\notin \Soc (A)$.
\end{exa}

\label{braces}

\section{Reflections from factorizable groups}
\label{factorization}

An obvious question is how the solutions of the YBE and reflection equation
that come from braces relate to previously known solutions. In this section we
consider the skew brace description of the YBE solutions of Weinstein and Xu,
and consider associated solutions of the reflection equation.  These YBE
solutions are also related to the triangular Hopf algebras obtained by Beggs,
Gould and Majid \cite{BGM96}

In~\cite{MR1178147} Weinstein and Xu produced set-theoretic solutions to the
YBE by using factorizable groups. Using the language of skew braces we use
group factorizations to construct reflections.

A skew left brace is a non-abelian generalization of a left brace \cite{MR3647970}. More precisely it is a set with two compatible group structures. One of these groups is known as the multiplicative group; the other as the additive group. The terminology used in the theory of Hopf-Galois extensions suggests that the additive group determines the type of the skew left brace. For example, skew left braces of abelian type are Rump braces, that is,  braces with an abelian additive group. A skew left brace is a triple $(A,+,\circ )$, where $(A,+)$ and $(A,\circ $) are (not necessarily abelian) groups such that the compatibility 
$a\circ (b + c)=a\circ b -a + a\circ c$ holds for all $a,b,c \in A$. For a skew left brace $A$, the map $r_{A}: A\times A \rightarrow  A\times A$, $r_{A}(a,b)=(-a + a\circ b,(-a + a\circ b)' \circ a\circ b)$  is a non-degenerate set-theoretic solution of the Yang--Baxter equation. We write $a'$ to denote the inverse of $a$ with respect to the circle operation $\circ $.

We say that an additive (and not necessarily abelian) group $G$ admits an exact
factorization through subgroups $A$ and $B$ if 
\[
G=A+B=\{a+b:a\in A,b\in B\}
\]
and $A\cap B=\{0\}$. This means that for $x\in G$ there are unique $a\in A$ and
$b\in B$ such that $x=a+b$. 

By~\cite[Theorem 3.3]{MR3763907} the group $G$ with circle operation $x\circ
y=a+y+b$ whenever $x=a+b\in AB$ is a skew left brace. By~\cite{MR3647970}, the
pair $(G,r_G)$, where
\[
	r_G\colon G\times G\to G\times G,\quad
	r_G(x,y)=(\lambda_x(y),\mu_y(x))
\]
is a non-degenerate solution to the YBE. Note that 
$x\circ y=\lambda_x(y)\circ \mu_y(x)$ for all $x,y\in G$. Moreover, if 
$x=a+b$, then 
\[
	\lambda_x(y)=-x+x\circ y=-b-a+a+y+b=-b+y+b
\]
for all $y\in G$.  We collect some useful formulas in the following lemma:

\begin{lem}
	Let $G$ be a group that admits an exact factorization through subgroups $A$
	and $B$ and $z=cd\in AB$ for $c\in A$ and $d\in B$ central elements of $G$.
	Then 
	\begin{align}
		&\lambda_{xz}(y)=\lambda_x(y),
		&& \lambda_x(yz)=\lambda_x(y)z,\\
		& x\circ (yz)=(a\circ b)z, 
		&& (xz)\circ y=(x\circ y)z
	\end{align}
	for all $x,y\in G$.
\end{lem}

\begin{proof}
	Write $x=ab$ with $a\in A$ and $b\in B$. Then
	\[
	\lambda_{xz}(y)=(bd)^{-1}y(bd)=b^{-1}yb=\lambda_x(y).
	\]
	The other formulas are proved similarly. 
\end{proof}

Now we prove the main result in this section:

\begin{thm}
	Let $G$ be a group that admits an exact factorization through subgroups $A$
	and $B$ and let $(G,r_G)$ be its associated solution to the YBE.  For elements 
	 $c\in A$ and $d\in B$ which are central in $G$ and such that $cd$ is central in $G$, 
	 the map $k(x)=x(cd)$ is a reflection of $(G,r_G)$.
\end{thm}

\begin{proof}
	Let $x=ab\in G$ with $a\in A$ and $b\in B$ and $y\in G$.
	We claim that $x\circ y=\lambda_x(y)\circ \mu_{yz}(x)$. Indeed, 
	\begin{align*}
		(x\circ y)z &= x\circ (yz) 
		= x\circ k(y)\\
		&=\lambda_x(yz)\circ \mu_{yz}(x)
		=(\lambda_x(y)z)\circ \mu_{yz}(x)
		=(\lambda_x(y)\circ \mu_{yz}(x))z.
	\end{align*}
	Therefore $\mu_{yz}(x)=\mu_y(x)$. This implies that 
	\begin{align*}
		rk_2(x,y)&=r(x,k(y))
		=r(x,yz)\\
		&=(\lambda_x(yz),\mu_{yz}(x))
		=(\lambda_x(y)z,\mu_y(x))
		=k_1r(x,y).
	\end{align*}
	By Theorem~\ref{thm:k2r=rk1}, $rk_1=k_2r$. Hence $k$ is a reflection by
	Proposition~\ref{pro:Agata2}.
\end{proof}

\label{factorizable}
\section{Parameter-dependent solutions}\label{paramdepsolns}
In the application to quantum integrable systems, the interest is usually in solutions to the following parameter-dependent quantum Yang--Baxter and reflections equations
\begin{align}
&\big(\mathfrak{r}(u)\otimes
\id\big)\big(\id\otimes \mr(u+v)\big)\big(\mr(v)\otimes \id\big)\nonumber\\&= \big(\id\otimes \mr(v)\big)\big(\mr(u+v)\otimes
\id)(\id\otimes \mr(u)\big),\label{eq:paramyb}\\		
&\big(\id\times \mk(v)\big)\mr(u+v) \big(\id\times \mk(u)\big)
 \mr(u-v)\nonumber\\
&=\mr(u-v) \big(\id\times \mk(u)\big) \mr(u+v) \big(\id\times \mk(v)\big),
\label{eq:parambyb}
\end{align}
where $u$ and $v$ are in $\mathbb{C}$. Here $\mr(u):V\otimes V \rightarrow V\otimes V$ and $\mk(u):V\rightarrow V$, where $V=\mathbb{C} X$ is the vector space spanned by elements of $X$.
These parameter dependent solutions $\mr(u)$ and $\mk(u)$ are used to define {\it transfer matrices} which are certain maps $T(u):V^{\otimes N} \rightarrow V^{\otimes N}$. The details may be found in the paper \cite{MR953215}, but the key point to note is that the properties \eqref{eq:paramyb} and \eqref{eq:parambyb} lead to the result that $T(u) T(v)= T(v) T(u)$ for all $(u,v)\in \mathbb{C}\times \mathbb{C}$. This commutativity is the key defining property of a quantum integrable system. In this section, we describe a very simple way in which rational 
parameter dependence may be introduced into our $r$ and $k$ matrices. 

Let $(X,r)$ be an involutive non-degenerate set-theoretic solution and write 
\[
r(x,y)=(\sigma _{x}(y), \tau _{y}(x)),\quad x,y\in X.
\]

Let $V=\mathbb C X$ be the linear space over the field of complex numbers spanned by
elements from $X$. For $u\in \mathbb C$ let 
\[
R(u):V\otimes V\rightarrow V\otimes V,
\quad
R(u)=\id +u r,
\]
where $r:V\otimes V\rightarrow V\otimes V$ is
defined by $r(x\otimes y)=\sigma _{x}(y)\otimes \tau _{y}(x)$ for $x,y\in X$. 

Notice that $R(u)$ is a rational solution of the (parameter dependent)     
Yang--Baxter equation \[
	(R(u)\otimes
\id)(\id \otimes R(u+v))(R(v)\otimes \id)= (\id\otimes R(v))(R(u+v)\otimes
\id)(\id\otimes R(v)).\]

Let $k: X\rightarrow X$ be a function (for example a reflection on $X$). We can extend this function linearly to
a linear map $k\colon V\to V$. For a map $k:V\rightarrow V$ and  fixed $u\in \mathbb C$ define map $K(u):V\rightarrow  V$ by $K(u)(x)=uk(x)$ for $x\in X$.

\begin{thm}
    \label{thm:parameter}
    Let $(X,r)$ be an involutive non-degenerate solution to the YBE, $k:X\rightarrow X$ be an involutive reflection of $(X,r)$, 
    $K_{1}(u) =K(u)\otimes id$ and $K_{2}(u) = id \otimes K(u)$.
    Then
    \[
    K_{2}(v)R(u+v)K_{2}(u)R(u-v)=R(u-v)K_{2}(u)R(u+v)K_{2}(v).
    \] 
\end{thm}

\begin{proof} 
 To prove the claim 
  it is sufficient to show that 
  \begin{align*}
   k_{2}R(u+v)k_{2}R(u-v)=R(u-v)k_{2}R(u+v)k_{2}.
  \end{align*}
 By using the fact that $R(u)=\id+ur$ we get that
  for each $x,y\in X$ we have 
  \begin{align*}
	k_{2}R(u+v)k_{2}R(u-v)&=k_{2}(\id+(u+v) r)k_{2}(\id+(u-v) r)\\
	&\hspace*{-15mm}=k_{2}k_{2}+(u-v)k_{2}k_{2}r+(u+v)k_{2}rk_{2}+(u+v)(u-v) k_{2}rk_{2}r
	\shortintertext{and}
	R(u-v)k_{2}R(u+v)k_{2}&=(\id+(u-v) r)k_{2}(\id+(u+v) r)k_{2}\\
	&\hspace*{-15mm}=k_{2}k_{2}+(u-v)rk_{2}k_{2}+(u+v)k_{2}rk_{2}+(u+v)(u-v)rk_{2}r k_{2}.
  \end{align*}
 Since $k$ is an involution, $k_{2}k_{2}r= r=rk_{2}k_{2}$. Moreover, 
 $rk_{2}rk_{2}=rk_{2}rk_{2}$. This concludes the
 proof. 
\end{proof} 

\begin{rem}
\label{rem:parameter}
In Theorem~\ref{thm:parameter} 
we can alternatively take 
\[
K(u)(x) = f(u)k(x),
\]
where $f:\C\rightarrow\C$ is an arbitrary function. 
\end{rem}


Some solutions to the reflection equation constructed earlier in this paper are involutive.
We now present three examples of involutive reflections:

\begin{example}
	Let $A$ be a nilpotent ring, $X\subseteq A$ be a subset such that the
	restriction $r$ of $r_A$ to $X\times X$ is a solution to the Yang--Baxter
	equation. Let $c\in A$ be such that $c*c=0$, define $k(x) = x + x*c$, then
	\[
		rk_{2}rk_{2} = k_{2}rk_{2}r.
	\]
	Moreover, if $a+a=0$ for every $a\in A$, then $k$ is an involution. This
	means that $k$ and $r$ give solutions to the parameter depended equation as
	in Theorem~\ref{thm:parameter}.
\end{example}

\begin{example} 
Let $A$ be a nilpotent ring 
and $X\subseteq A$ be a subset such that the
restriction $r$ of $r_A$ to $X\times X$ is a solution to the Yang--Baxter
equation. 
Define $k(x) = x + x^{n-1}$, then
$rk_{2}rk_{2} = k_{2}rk_{2}r$.  Moreover, if $a+a=0$ for every $a\in A$, then
\[
k^{2}(x)=k(x+x^{n-1})=x+x^{n-1}+x^{n-1}=x.
\]
Hence $k$ is an involution. So
$k$ and $r$ give solutions to the parameter dependent equation as in
Theorem~\ref{thm:parameter}.
\end{example}

\begin{example}
	Let $A$ be a nilpotent ring such that $a+a=0$ for every $a\in A$. Let  $I$ be an ideal of $A$
	and $g\colon A\to A$ be such that
	$g(A)\subseteq I$ and $g(a+x)=g(a)$ for all $a\in A$ and $x\in I$. Assume that $g(x)^{2}=0$ for every $x\in A$. Then $k(x)=x+x*g(x)$ is an involution and a reflection of $(A,r_A)$.  So
$k$ and $r$ give solutions to the parameter dependent equation as in
Theorem~\ref{thm:parameter}.
\end{example}

Notice that the linear mapping $K$ from Theorem~\ref{thm:parameter}, when translated to a matrix form will have exactly one non-zero entry in each column, because it comes from a set-theoretic solution to YBE.  
Since  the setting of Theorem~\ref{thm:parameter} is no longer set-theoretic, is natural to ask 
to find $K$-matrices $k:V\rightarrow V$  which have many nonzero entries in their rows and columns.
We notice that a  small  modification of Theorem~\ref{thm:aggi} gives large classes of such $K$-matrices.

\begin{thm} 
\label{thm:Agata_new}
Let $A$ be a two-sided brace, $X\subseteq A$ be a subset such that the restriction  $r=r_A|_{X\times X}$ is a solution $(X,r)$ to the YBE. 
Let $g_1,\dots,g_n\in\C$ and $f_1,\dots,f_n\colon X\to X$ be $\mathcal{G}(X,r)$-equivariant maps.
Let $V=\C X$ be the vector space with basis in the elements of $X$ and denote by
$\oplus$ the addition of $V$.
Let $k:V\to V$ defined as  
 \[
 k(x)= f _{1}(x)\cdot  g_{1} \oplus  f _{2}(x)\cdot g_{2}\oplus \ldots  \oplus f _{n}(x)\cdot  g_{n}
 \] 
for $x\in X$ and then extended by linearity to $V$. If  $k(x)$ is involutive, then $K(u)=uk$ satisfies the reflection equation 
\[
K_{2}(v)R(u+v)K_{2}(u)R(u-v)=R(u-v)K_{2}(u)R(u+v)K_{2}(v),
\] 
where $R(u)=\id+ur$ and $K_{2}(u)=id \otimes K(u)$.
\end{thm}

\begin{proof}
To prove the claim it is sufficient to show that 
\[
K_{2}R(u + v)K_{2}R(u-v) = R(u-v)K_{2}R(u + v)K_{2}.
\]
By using the fact that $R(u) = id+ur$ we get that for each $x,y \in  X$ we have 
\begin{multline*}
K_{2}R(u + v)K_{2}R(u-v) = K_{2}(id + (u + v)r)K_{2}(id + (u-v)r)\\
=K_{2}K_{2} + (u-v)K_{2}K_{2}r + (u + v)K_{2}rK_{2} + (u + v)(u-v)K_{2}rK_{2}r.
\end{multline*}
Notice also that 
\begin{multline*}
R(u-v)K_{2}R(u + v)K_{2} = (id + (u-v)r)K_{2}(id + (u + v)r)K_{2} \\
= K_{2}K_{2} + (u-v)rK_{2}K_{2} + (u + v)K_{2}rK_{2} + (u + v)(u-v)rK_{2}rK_{2}.
\end{multline*}
Since $K$ is an involution, $K_{2}K_{2}r = r = rK_{2}K_{2}$. It remains to show that 
\[
rK_{2}rK_{2} = K_{2}rK_{2}r.
\]
We calculate the left hand side for $x,y\in X$: 
\begin{align*}
rK_{2}rK_{2}(x,y)&=rK_{2}r(x, g _{1} f_{1}(y) \oplus  g_{2}f_{2}(x)\oplus \ldots  \oplus g _{n}  f_{n}(y))\\
&=\oplus _{i=1}^{n} g_{i}rK_{2}r(x, f_{i}(y))\\
&=\oplus _{i=1}^{n} g_{i}rK_{2}(\sigma _{x}( f_{i}(y))\otimes \tau _{f_{i}(y)}(x))\\
&=\oplus _{j=1}^{n}\oplus _{i=1}^{n} g_{j}g_{i}(c_{i,j}\otimes d_{i,j}),
\end{align*}
where 
\[
c_{i,j}=\sigma _{a}(f_{j}(b)),
\quad 
d_{i,j}=\tau _{f_{j}(b)}(a),
\quad
a=\sigma _{x}( f_{i}(y)),
\quad
b=\tau _{f_{i}(y)}(x).
\]
Let $A^{1}$ be natural extension of the ring $A$ by an identity element, then \[d_{i,j}=(1+c_{i,j})^{-1}*a=(1+c_{i,j})^{-1}*\sigma _{x}( f_{i}(y))\] where $(1+c_{i,j})^{-1}*(1+c_{i,j})=1.$
 Observe that because each $f_{i}(x)$ is $\mathcal {G}(X,r)$- equivariant, we get 
\[c_{i,j}=\sigma _{a}(f_{j}(b))= f_{j}(\sigma _{a}(b))=f_{j}(x)\] since $r$ is involutive.  
 Therefore \[d_{i,j}=(1+c_{i,j})^{-1}*a=(1+c_{i,j})^{-1}*\sigma _{x}( f_{i}(y))=(1+c_{i,j})^{-1}*f_{i}(\sigma _{x}(y)).
 \]

We now calculate the left hand side of our equation for $x,y\in X$:
\begin{align*}
K_{2}rK_{2}r(x,y) &= K_{2}rK_{2}(\sigma _{x}(y), \tau _{y}(x))\\
&=\oplus _{j=1}^{n}g_{j}K_{2}r(\sigma _{x}(y), f_{j}(\tau _{y}(x))\\
&=\oplus _{i=1}^{n}\oplus _{j=1}^{n}g_{i}g_{j}( e_{i,j}\otimes q_{i,j}),
\end{align*}
where 
\[
e_{i,j}=\sigma _{p}(f_{j}(q)),
\quad
q_{i,j}=f_{i}(\tau_{f_{j}(q)}(p)),
\quad
p=\sigma _{x}(y),
\quad
q=\tau_{y}(x).
\]
Notice that  because $f_{j}$ is $\mathcal {G}(X,r)$-equivariant we get 
\[
e_{i,j}=\sigma _{p}(f_{j}(q))=f_{j}(\sigma _{p}(q))=f_{j}(x).
\]
It follows that $e_{i,j}=c_{i,j}$.

We will now calculate $q_{i,j}$: 
We know that $\tau_{f_{j}(q)}(p)=(e_{i,j}+1)^{-1}*p$ where $(e_{i,j}+1)^{-1}*(e_{i,j}+1)=1$. We now calculate 
\[
q_{i,j}=f_{i}(\tau_{g_{j}(q)}(p))=f_{i}((e_{i,j}+1)^{-1}*p)=(c_{i,j}+1)^{-1}*f_{i}(p).
\]
Since $p=\sigma _{x}(y)$, we get $q_{i,j}=d_{i,j}$ and the result follows. 
\end{proof}

\begin{rem}
We can guarantee that  the above map $k$ is involutive by assuming the appropriate relations on the underlying ring $A$  and on elements $f_{i}$. Notice that the matrix associated to the map $K:V\rightarrow V$ will usually have around $n$ nonzero entries in each column.
\end{rem}

\begin{rem}
Let $(X,r)$ and $k:V\rightarrow V$ be constructed as in Theorem 5.6, then $k$ is a solution to the set theoretic reflection equation 
\[K_{2}rK_{2}r=rK_{2}rK_{2}\] where $K_{2}=id \otimes k$.
\end{rem}

\begin{example}
Let $A$ be a nilpotent ring such that the sum of three copies of any element from this ring is zero. Let $c\in A$ be such that $c*c=0$ and 
\begin{align*}
& c_{1}=c, &&c_{2}=2c, && c_{3}=3c=0. 
\end{align*}
For each $i\in\{1,2,3\}$ let $f_{i}(x)=x*c_{i}+x$
and 
\[
L(x)=(2/3)f_1\oplus (2/3)f_2+(-1/3)f_3.
\]
A straightforward calculation shows that $L(L(x))=x$. Then $L$ is an involutive map which satisfies the assumptions 
of Theorem~\ref{thm:Agata_new}, so it gives a solution to a parameter dependent reflection equation.
\end{example}

\begin{rem}
In Theorem~\ref{thm:Agata_new} we can take functions $f_{i}: V\rightarrow V$  defined as follows: for a given $i$ let either $f_{i}(x)=x*c_{i}$ or $f_{i}(x)=x+x*c_{i}$  for $x\in X$ and then extend them by linearity to $V$. 
\end{rem}

\begin{thm}
\label{thm:Robert}
Let $(X,r)$ be an involutive non-degenerate solution to the YBE, let $V=\C X$ and 
denote by $r$ the linear extension of $r$ to $V\otimes V$.  
Let  $ k\colon X\rightarrow  X$ be a reflection of $(V,r)$ so $rK_{2}rK_{2}=K_{2}rK_{2}r$, where $K_{2}=\id\otimes k$. Assume that either $k^{2}=\id$ or $k^{2}=0$.
Denote $K_{2}(u)=\id \otimes K(u)$ and $R(u)=\id+u{r}$, where  \[K(u)=\id +uk.\]
  Then $K(u)$ satisfies the reflection equation  
  \[
  K_{2}(v)R(u + v)K_{2}(u)R(u-v) = R(u-v)K_{2}(u)R(u + v)K_{2}(v).
  \]
\end{thm}

\begin{proof} 
Denote $C(u,v)=K_{2}(v)R(u+v)K_{2}(u)$. We have to show that 
\[C(u,v)R(u-v)=R(u-v)C(v,u).\]
   This is equivalent to show that 
\[C(u,v)-C(v,u)+(u-v)(C(v,u)r-rC(v,u))=0,\] since $R(u-v)=\id+(u-v)r$.

Because all maps are $\C$-linear we have $K_{2}(u)=\id+uK_{2}$. 
Notice that 
\begin{multline*}
    C(u,v)=v(u+v)K_{2}r+u(u+v)rK_{2}\\
+\id+(u+v)(K_{2}+r)+uv(u+v)K_{2}rK_{2}+uvK_{2}^{2}.
\end{multline*}
Therefore 
\[
C(u,v)-C(v,u)=(v^{2}-u^{2})(K_{2}r-rK_{2}).
\]

A straightforward calculation shows that 
\[
C(u,v)r-rC(v,u)=
(u+v)(K_{2}r-rK_{2})
\]
since by assumption $rK_{2}rK_{2}=K_{2}rK_{2}r$ and $K_{2}^{2}r=rK_{2}^{2}$.
Therefore \[
C(u,v)-C(v,u)+(u-v)(C(v,u)r-rC(v,u))=0,
\]
as required.
\end{proof}
%

\begin{rem}
Notice that conditions from Theorem~\ref{thm:Robert} include conditions $K(\lambda )K(-\lambda ) = \id$ for each $\lambda \in C$, and $K(0) =\id$, from~\cite{MR1965153} provided that $k^{2}=0$. On the other hand if
 $k^{2}=\id$ then we can define $K(u)=\frac {\id+ku}{\sqrt {1-u^{2}}}$ for $u\neq 1, -1$ and the conditions from~\cite{MR1965153} are satisfied.
\end{rem}

Notice that Theorem~\ref{thm:Agata_new} also holds if we assume that $k^{2}=0$ instead of the assumption $k^{2}=\id$ and it gives a rich source of maps $k$ which could be used in Theorem~\ref{thm:Robert}.

\section*{Acknowledgements}
Smoktunowicz is supported by 
ERC Advanced grant 320974 and  
EPSRC Programme Grant EP/R034826/1. 
Vendramin acknowledges the support of NYU-ECNU Institute of Mathematical Sciences at NYU Shanghai, UBACyT 20020171000256BA and PICT-2481-0147. 
Weston acknowledges support from the EPSRC grant: EP/R009465/1 and would like
to thank Anastasia Doikou and Bart Vlaar for useful discussions.

\bibliographystyle{abbrv}
\bibliography{refs}

\end{document}